\documentclass[11pt]{amsart}
\usepackage{amssymb,latexsym}
\usepackage{amsmath,amsfonts,amsthm}
\usepackage[dvips]{graphicx}
\usepackage[all]{xy}
\newtheorem{theorem}{Theorem}[section]
\newtheorem{prop}[theorem]{Proposition}
\newtheorem{assm}[theorem]{Assumption}
\newtheorem{lemma}[theorem]{Lemma}

\newtheorem{remark}[theorem]{Remark}

\newtheorem{definition}[theorem]{Definition}
\newtheorem{cor}[theorem]{Corollary}

\def\co{\colon\thinspace}

\def\barnu{\bar{\nu}}
\def\ep{\epsilon}

\begin{document}

\title{Duality in filtered Floer--Novikov complexes}

\author{Michael Usher}\address{Department of Mathematics, University of Georgia, Athens, GA 30602}\email{usher@math.uga.edu}

\begin{abstract}
We prove that a certain  bilinear pairing (analagous to the Poincar\'e--Lefschetz intersection pairing) between  filtered sub- and quotient complexes of a Floer-type chain complex and of its ``opposite complex''
is always nondegenerate on homology.  This implies a duality relation for the Oh--Schwarz-type spectral invariants of these complexes which (in Hamiltonian Floer theory) was established in the special case that the period map has discrete image by Entov and Polterovich.  The duality relation  
served as a key lemma in Entov and Polterovich's construction of a Calabi quasimorphism on certain rational symplectic manifolds, and the result that we prove here implies that their construction remains valid when the rationality hypothesis is dropped.  Apart from this, we also use the nondegeneracy of the pairing to establish a new formula for what we have previously called the boundary depth of a Floer chain complex; this formula shows that the boundary depth is unchanged under passing to the opposite complex.

\end{abstract}
\maketitle
Given a Morse function $f\co M\to\mathbb{R}$ on a smooth closed manifold together with a suitable Riemannian metric, a now-standard construction produces a chain complex $CM_*(f)$, freely generated by the critical points of $f$ and with a boundary operator whose matrix elements are counts of isolated negative gradient flowlines of $f$; the simple chain homotopy type of the complex is independent of $f$, and its homology is isomorphic to the singular homology of $M$.  \emph{Poincar\'e duality} has an especially simple description in terms of the Morse chain complex:  the dual chain complex to $CM_*(f)$ can be obtained by having the boundary operator count positive gradient flowlines of $f$ instead of negative ones, so the fact that a positive gradient flowline of $f$ is equivalent to a negative gradient flowline of $-f$  identifies the dual complex to $CM_*(f)$ with $CM_{\dim M-*}(-f)$, from which Poincar\'e duality immediately follows.

Suppose now that $\alpha$ is a regular value of $f$, and let $M_{\alpha}=f^{-1}((-\infty,\alpha))$.  One can then restrict attention to the part of the complex $CM_*(f)$ generated by those critical points with $f$-value at most $\alpha$; this is a subcomplex by virtue of the fact that $f$ decreases along its negative gradient flowlines.  This complex, $CM_{*}^{(-\infty,\alpha)}(f)$, then has homology isomorphic to $H_*(M_{\alpha};\mathbb{Z})$.  
On the other hand, since a \emph{positive} gradient flowline of $f$ which begins in $M_{\alpha}$ is likely to leave $M_{\alpha}$, the dual complex to 
$CM_{*}^{(-\infty,\alpha)}(f)$ is not a subcomplex of $CM_{\dim M-*}(-f)$; on the contrary, it's not hard to see that 
the dual complex to $CM_{*}^{(-\infty,\alpha)}(f)$ is naturally identified with the quotient complex $CM_{\dim M-*}^{(-\alpha,\infty)}(-f):=CM_{\dim M-*}(-f)/CM_{\dim M-*}^{(-\infty,-\alpha)}(-f)$.\footnote{For simplicity in these introductory remarks we assume that $\alpha$ is a regular value of $f$, so that there's no distinction between, \emph{e.g.}, $CM_{*}^{(-\infty,-\alpha)}(-f)$ and $CM_{*}^{(-\infty,-\alpha]}(-f)$ in the obvious notation.  Later in the paper we will drop this assumption on $\alpha$ and pay closer attention to such distinctions, which will become material when we consider functions with a dense set of critical values.}  One can show that $CM_{*}^{(-\alpha,\infty)}(-f)$ models the relative chain complex of the pair $\left(M,f^{-1}([\alpha,\infty))\right)$, and so by excision the homology of $CM_{*}^{(-\alpha,\infty)}(-f)$ is isomorphic to $H_{*}(\overline{M_{\alpha}},\partial\overline{M_{\alpha}};\mathbb{Z})$.  Thus the pairing between $CM_{*}^{(-\infty,\alpha)}(f)$ and its dual complex expresses the Poincar\'e--Lefschetz duality pairing $H_{*}(M_{\alpha};\mathbb{Z})\times  H_{\dim M-*}(\overline{M_{\alpha}},\partial\overline{M_{\alpha}};\mathbb{Z})\to\mathbb{Z}$.

This paper is concerned with generalizations of this Morse-theoretic Poincar\'e--Lefschetz pairing to the setting of chain complexes such as those encountered in Novikov's Morse theory for closed one-forms and in Floer theory.  To elucidate the principles at work here, we work in an abstract algebraic setting (a ``graded filtered Floer--Novikov complex'') which models both the Novikov and Floer situations.  The relationship above between $CM_*(f)$ and $CM_*(-f)$ is abstracted into the notion of the ``opposite complex'' of a graded filtered Floer--Novikov complex.  In the case modeling that of the Novikov complex of a closed one-form which is obtained as the derivative of an $S^1$-valued Morse function, it follows quickly from the definitions that the filtered complex analagous to the above $CM_{*}^{(-\infty,\alpha)}(f)$ is the dual of the one analagous to $CM_{*}^{(-\alpha,\infty)}(-f)$, which immediately gives rise to a Poincar\'e--Lefschetz-type duality between their respective homologies.  The more interesting (``irrational'') case is that which models the Novikov complex of a closed one-form whose periods form a dense subgroup of $\mathbb{R}$.  In this case, there continues to be a natural nondegenerate\footnote{In this paper a nondegenerate pairing between two vector spaces $U$ and $V$ over a field $K$ is by definition a bilinear map $\langle\cdot,\cdot\rangle\co U\times V\to K$ such that for all  $u\in U\setminus\{0\}$ and $v\in V\setminus\{0\}$ neither the map $\langle u,\cdot\rangle\co V\to K$ nor the map $\langle \cdot,v\rangle\co U\to K$ is identically zero.  Of course, in finite dimensions but not in infinite dimensions this identifies $U$ and $V$ as each other's duals; consequently there is a little bit of inconsistency in the literature as to what ``nondegeneracy'' means in the infinite-dimensional context, but we will use the above convention throughout.} pairing between the (infinite rank) complexes analagous to $CM_{*}^{(-\infty,\alpha)}(f)$ and $CM_{*}^{(-\alpha,\infty)}(-f)$; however this pairing only identifies the respective complexes as \emph{subcomplexes} of each other's dual complexes.  Nonetheless, one still gets a pairing on the level of homology, and our main theorem (Theorem \ref{nondeg}) asserts that this homological pairing is 
still nondegenerate (when we work with field coefficients, thus precluding issues of torsion).  Hence in the irrational case, although one does not have a Poincar\'e--Lefschetz duality \emph{isomorphism} as holds in the Morse case, some sort of Poincar\'e--Lefschetz duality survives in the form of a nondegenerate ``intersection pairing.''

Our main motivation in proving Theorem \ref{nondeg} is to obtain Corollary \ref{spectraldual}, which expresses a duality relationship between the spectral numbers (as in \cite{Ucomp}) of a graded filtered Floer--Novikov complex and of its opposite complex; in a similar vein, we also use Theorem \ref{nondeg} to prove a duality result (Corollary \ref{bdepth}) for the boundary depth (as in \cite{U09}).  In the case of Hamiltonian Floer homology, the original complex contains information about a certain nondegenerate Hamiltonian isotopy, and then its opposite complex encodes information about the inverse isotopy; thus Corollaries \ref{spectraldual} and \ref{bdepth} express  relationships between, respectively, the Oh--Schwarz spectral invariants (\cite{Sc00},\cite{Oh}) and the boundary depths of any two nondegenerate Hamiltonian isotopies which are each other's inverses.  In the ``rational'' case (corresponding formally to a $1$-form arising from an $S^1$-valued Morse function) the duality relationship for the spectral invariants was established (generalizing \cite[Proposition 4.9]{Sc00}) by Entov and Polterovich \cite{EP}, who used it to remarkable effect, proving that the Oh--Schwarz spectral invariants give rise to Calabi quasimorphisms on the universal cover of the Hamiltonian diffeomorphism group of any closed rational semi-positive symplectic manifold $(M,\omega)$ whose quantum homology in grading $\dim M$  contains a field as a direct summand.\footnote{In \cite{EP} a somewhat stronger requirement is imposed; see \cite{EP2} for the version alluded to here.}  In the irrational case, corresponding to the situation where the symplectic form $\omega$ on $M$ has integrals over spheres which comprise a dense subgroup of $\mathbb{R}$, as far as the author is aware there are no analogues (either in the Hamiltonian  Floer case or more generally) of Corollary \ref{spectraldual} that have been established until now.  

The arguments in \cite[Section 5]{Ost},\cite{EP2} use the rationality of $\omega$ only in two ways: one is to ensure that the ``nondegenerate spectrality'' axiom holds for the spectral invariants, but this is now known in the irrational case due to \cite[Corollary 1.5]{Ucomp}; and the other  is to ensure that the Hamiltonian Floer version of Corollary \ref{spectraldual} (\emph{i.e.}, \cite[Lemma 5.1]{Ost}) holds.  Consequently it follows from Corollary \ref{spectraldual} together with these other references that (see \cite{EP2} for definitions):
\begin{cor} On any closed semi-positive\footnote{In fact, if one is willing to use the machinery of virtual moduli cycles, the semi-positivity assumption here is unnecessary: the virtual cycle-based construction of Hamiltonian Floer theory satisfies the algebraic axioms considered here and in \cite{Ucomp}, and results of \cite{Lu04} can be used to see that the various properties that are used in Entov and Polterovich's work continue to hold in the non-semi-positive case.} symplectic manifold $(M,\omega)$ whose degree-$\dim M$ quantum homology contains a field direct summand,  $-\int_M\omega^n$ times the homogenization of the spectral invariant associated to the identity in the field summand gives a Calabi quasimorphism on the universal cover of the Hamiltonian diffeomorphism group.\end{cor}  (Moreover, from this result one gets a symplectic quasi-state on $M$ as described in \cite{EPqs},\cite{EP2}).    The assumption that the quantum homology of $(M,\omega)$ admits a field direct summand, while satisfied in a number of interesting cases (see \cite{EP2}), is a significant restriction, and seems to be really necessary to obtain results such as these (for instance, see the discussion in \cite[p. 84]{EPqs}).  

While we are on the topic, a clarification is perhaps in order.  In \cite{Ucomp} it was noted the the construction of a ``\emph{partial} symplectic quasi-state'' from \cite{EPqs} on any rational closed symplectic manifold in fact carried over to irrational manifolds by virtue of the main result of \cite{Ucomp}.  While this statement in \cite{Ucomp} is true, its phrasing may have caused some confusion, as since that time some authors appear to have construed it as meaning that all of the constructions of \cite{EPqs}, including those involving (genuine, not just partial) quasi-states and quasimorphisms, carried over to the irrational case.  This is not quite accurate, because Entov--Polterovich's method for establishing the quasimorphism property is crucially dependent on a version of Corollary \ref{spectraldual}, which until the present paper was only known in the rational case; the discussion of \emph{partial} quasi-states in \cite{EPqs}, on the other hand, does not rely on any such result.

The paper is organized into two sections. The first introduces the algebraic setup (abstracted from Floer theory) that we use and gives the precise statement of the main theorem (Theorem \ref{nondeg}) as well as  statements and proofs of the applications to spectral numbers (Corollary \ref{spectraldual}) and boundary depth (Corollary \ref{bdepth}).  The second section gives a proof of Theorem \ref{nondeg}, using only rather elementary arguments about linear algebra over Novikov rings with some input from \cite{GG} and \cite{Ucomp}.  

\subsection*{Acknowledgements}  The impetus for much of this work was provided by a discussion with Stefan M\"uller; I thank him and also NIMS and Seoul National University for their hospitality at the International Conference on Symplectic, Contact, and Low-Dimensional Topology at which the discussion took place.  Thanks also to Leonid Polterovich for his comments on an earlier draft.

\section{Formalism and statements of main results}

\subsection{Graded filtered Floer--Novikov complexes}\label{notation}

We work in the general algebraic setting of \cite{Ucomp}, except that it will be more convenient to have the Novikov ring over which the Floer complex is defined be a field.  In typical applications, this can arranged to be so modulo the minor additional complication that, while the entire complex is most naturally a module over a Novikov ring which is not a field, the complex carries a grading, and the subring of the Novikov ring which preserves the grading is a field.  Accordingly we are led to the notion of a \emph{graded filtered Floer-Novikov complex} $\frak{c}$ over a field $K$, consisting of data $\frak{c}=(\pi\co P\to S,\mathcal{A},\omega,gr,deg,\partial)$ as follows:

\begin{enumerate}
\item  $\pi\co P\to S$ is a principal $\Gamma$-bundle, where $\Gamma$ is a finitely-generated abelian group (written multiplicatively) and $S$ is a finite set.
\item The ``action functional'' $\mathcal{A}\co P\to\mathbb{R}$ and the ``period homomorphism'' $\omega\co \Gamma\to\mathbb{R}$ are related by \[ \mathcal{A}(g\cdot p)=\mathcal{A}(p)-\omega(g) \quad (g\in\Gamma,p\in P).\]
\item The ``grading'' $gr\co P \to\mathbb{Z}$ and the ``degree homomorphism'' $deg\co \Gamma \to\mathbb{Z}$ obey \[ gr(g\cdot p)=gr(p)-deg(g) \quad (g\in\Gamma,p\in P).\]\end{enumerate}

 Before defining $\partial$, we introduce additional notation.  For $k\in \mathbb{Z}$, define the degree-$k$ part of the \emph{Floer chain complex} to be \[ C_k(\mathfrak{c})=\left\{\left.\sum_{p\in P,gr(p)=k}a_p p\right|a_p\in K,(\forall C\in\mathbb{R})(\#\{p\in P|a_p\neq 0,\mathcal{A}(p)>C\}<\infty)\right\}\] Also, let $\Gamma_0=\ker(deg\co \Gamma\to \mathbb{Z})$ and define the degree-zero part of the Novikov ring to be \[ \Lambda=\left\{\left.\sum_{g\in\Gamma_0}b_g g\right|b_g\in K,(\forall C\in\mathbb{R})(\#\{g\in \Gamma_0|b_g\neq 0,\omega(g)<C\}<\infty)\right\}\] and, where $e$ is an arbitrary element of $\Gamma$ such that $deg(e)$ generates $deg(\Gamma)\leq\mathbb{Z}$, define the full Novikov ring to be $\bar{\Lambda}=\oplus_{n\in\mathbb{Z}}e^n\Lambda$. (If the homomorphism $deg$ vanishes, just put $\bar{\Lambda}=\Lambda$.)  $\Lambda$ is a completion of the group $K$-algebra of $\Gamma_0$, and so the $\Gamma$-action on $P$  induces an action of $\Lambda$ on $C_k(\frak{c})$ for each $k$, and an action of $\bar{\Lambda}$ on the full Floer complex $C_*(\frak{c})=\oplus_{k\in\mathbb{Z}}C_k(\frak{c})$ (with the summand $e^n\Lambda$ lowering the grading $k$ by $ndeg(q)$).  Finally, define a map \[ \ell\co C_*(\frak{c})\to\mathbb{R}\cup\{-\infty\}\mbox{ by }\ell\left(\sum_{p\in P}a_pp\right)=\max\{\mathcal{A}(p)|a_p\neq 0\}.\]  Thus $\ell(c)=-\infty$ if and only if $c=0$.  We can now supply the last ingredient in the definition of a graded filtered Floer-Novikov complex: \begin{enumerate} \item[(4)] The \textbf{Floer boundary operator} $\partial$ is a $\bar{\Lambda}$-module endomorphism \[ \partial\co C_*(\frak{c})\to C_*(\frak{c})\] obeying \[ \partial^2=0,\quad (\forall c\in C_*(\frak{c}))(\ell(\partial c)<\ell(c)), \quad \partial(C_k(\frak{c}))\subset C_{k-1}(\frak{c}).\]  
\end{enumerate}

Let $\frak{c}=(\pi\co P\to S,\mathcal{A},\omega,gr,deg,\partial)$ be a graded filtered Floer-Novikov complex.  Write $S=\{s_1,\ldots,s_{|S|}\}$, and let $p_1,\ldots,p_{|S|}\in P$ be such that $\pi(p_i)=s_i$.  It follows from the definition that, as a $\bar{\Lambda}$-module, $C_*(\frak{c})$ is the free $\bar{\Lambda}$-module generated by the $p_i$.  The Floer boundary operator $\partial$ is thus determined by its values on the $p_i$.  

We will have, for $p\in P$, \[ \partial p=\sum_{q\in P}n(p,q)q,\] where $n(p,q)\in K$ are constrained to satisfy: \begin{itemize}\item 	For all $g\in \Gamma$, $n(gp,gq)=n(p,q)$ (because $\partial$ is a $\bar{\Lambda}$-module endomorphism);
\item If $n(p,q)\neq 0$ then $\mathcal{A}(q)<\mathcal{A}(p)$ and $gr(q)=gr(p)-1$ (because, respectively, $\ell(\partial p)<\ell(p)$ and $\partial(C_{gr(p)}(\frak{c}))\subset C_{gr(p)-1}(\frak{c})$).
\item For all $C\in\mathbb{R}$, there are just finitely many $q$ with $\mathcal{A}(q)\geq C$ and $n(p,q)\neq 0$.\end{itemize} 

Examples of graded filtered Floer-Novikov complexes are supplied by most Floer theories and by  Novikov's Morse theory for closed one-forms.  In general, one has a closed one-form on a (possibly infinite-dimensional) manifold, and the set $S$ in the definition above is the zero locus of this one-form.  The main motivating examples for this note are the complexes arising from Hamiltonian Floer theory on a closed symplectic manifold $(M,\omega)$.  Here the set $S$ corresponds to the contractible one-periodic orbits of a Hamiltonian system on $M$, which are viewed as the zeros of a certain one-form $\frak{a}$ on the space $\mathcal{L}_0M$ of contractible loops in $M$.  The form  $\frak{a}$ is closed but typically not exact; its lift to a certain abelian cover $\pi\co \widetilde{\mathcal{L}_0M}\to \mathcal{L}_0M$ (with deck transformation group $\Gamma$ equal to a certain quotient of $\pi_2(M)$) is exact, and there is a standard choice of primitive $\mathcal{A}\co \widetilde{\mathcal{L}_0M}\to\mathbb{R}$ for $\pi^*\frak{a}$.  In our definition, the principal bundle $\pi\co P\to S$ is just the restriction of the covering space  $\widetilde{\mathcal{L}_0M}\to \mathcal{L}_0M $ over the finite zero locus $S$ of $\frak{a}$; thus $P$ consists of the critical points of $\mathcal{A}$.  The period homomorphism is obtained by evaluating the cohomology class of the symplectic form $\omega$ on classes in the relevant quotient $\Gamma$ of $\pi_2(M)$.  A natural grading $gr$ is provided by the Conley--Zehnder index, and the degree homomorphism $deg$ is given by $deg(A)=2\langle c_1(TM),A\rangle$ for $A\in \Gamma$.  Finally, $\partial$ is constructed by counting formal negative gradient flowlines of $\mathcal{A}$ (which are cylindrical solutions to a perturbed Cauchy--Riemann equation in $M$) in a standard way.  The construction originates (in special cases not requiring Novikov rings) in \cite{F}; see \cite{Sal} for a survey of the more general case.    

\subsection{Filtrations}

Because of the requirement that the differential in a graded filtered Floer-Novikov complex $\frak{c}$ must satisfy $\ell(\partial c)<\ell(c)$, if $\alpha\in\mathbb{R}$ the subsets \[ C_{*}^{(-\infty,\alpha)}(\mathfrak{c})=\{c\in C_*(\frak{c})|\ell(c)<\alpha\} \mbox{ and } C_{*}^{(-\infty,\alpha]}(\mathfrak{c})=\{c\in C_*(\frak{c})|\ell(c)\leq\alpha\} 
\] are preserved by the boundary operator $\partial$ of $C_*(\frak{c})$. Note that if the period homomorphism $\omega$ is nontrivial  $C_{*}^{(-\infty,\alpha)}(\mathfrak{c})$ and $C_{*}^{(-\infty,\alpha]}(\mathfrak{c})$ will not be submodules over the Novikov rings $\Lambda$ or $\bar{\Lambda}$, but they are certainly (infinite rank) $K$-submodules (and hence $K$-subcomplexes).

These filtrations have become useful in Hamiltonian Floer theory by virtue of the fact that, whereas the chain homotopy type of the entire chain complex $C_*(\frak{c})$ is independent of the choice of Hamiltonian system on a given symplectic manifold, the filtered chain complexes $C_{*}^{(-\infty,\alpha)}(\mathfrak{c})$ do depend on the choice and can provide interesting information about the dynamics of the particular system.

\subsection{The opposite complex}

If $\frak{c}=(\pi\co P\to S,\mathcal{A},\omega,gr,deg,\partial)$ is a graded filtered Floer-Novikov complex, we can construct a new graded filtered Floer-Novikov complex $\frak{c}^{op}=(\pi^{op}\co P\to S,\mathcal{A}^{op},\omega^{op},gr^{op},deg^{op},\delta)$, the \emph{opposite complex of} $\frak{c}$, as follows:
\begin{enumerate}
\item Set theoretically, the map $\pi^{op}\co P\to S$ is the same as $\pi$; however in $\frak{c}^{op}$ the group $\Gamma$ acts on $P$ by $g\cdot^{op}p=(g^{-1})\cdot p$.
\item $\mathcal{A}^{op}=-\mathcal{A}$, and $\omega^{op}=\omega$.
\item $gr^{op}=-gr$, and $deg^{op}=deg$.
\item As noted earlier, the differential $\partial$ for the graded filtered Floer-Novikov complex $\frak{c}$ is determined by the values $n(p,q)\in K$ for which \[ \partial p=\sum_{q\in P}n(p,q)q.\]  Define \[ \delta\co C_*(\frak{c}^{op})\to C_*(\frak{c}^{op})\] by \[ \delta p=\sum_{q\in P}n(q,p)q.\]
\end{enumerate}

We leave it as an exercise for the reader to verify that these definitions indeed make $\frak{c}^{op}$ into a graded filtered Floer-Novikov complex. (A reader who has difficulty checking that $\delta$ satisfies the required properties might take hints from the proof of Proposition \ref{adjoint} below.)  

In the motivating case of Hamiltonian Floer theory, if the initial graded filtered Floer-Novikov complex $\frak{c}$ corresponds to a Hamiltonian flow $\{\phi_{H}^{t}\}_{0\leq t\leq 1}$, the opposite complex $\frak{c}^{op}$ corresponds to the \emph{inverse} Hamiltonian flow $\{(\phi_{H}^{t})^{-1}\}_{0\leq t\leq 1}$.  Indeed, an element of the underlying set $P$ for $\frak{c}$ corresponds to a contractible one-periodic orbit $\gamma$ for the initial flow $\phi_{H}^{t}$ together with an equivalence class of discs with boundary $\gamma$; orientation reversal of both the orbit and the disc puts such elements in one-to-one correspondence with the set of similar such data for $(\phi_{H}^{t})^{-1}$, in a way which reverses the action of the deck transformation group and negates both the action functional and the Conley--Zehnder index.  Moreover, negative gradient flowlines which flow from $p$ to $q$ from the standpoint of $\phi_{H}^{t}$ are easily seen to be in bijection with negative gradient flowlines from $q$ to $p$   from the standpoint of $(\phi_{H}^{t})^{-1}$, leading to the conclusion that the differential for the graded filtered Floer-Novikov complex associated to $(\phi_{H}^{t})^{-1}$ is given by $\delta$ as described in the above definition.

\subsection{Pairings}\label{pairings}  Fix a graded filtered Floer-Novikov complex $\frak{c}=(\pi\co P\to S,\mathcal{A},\omega,gr,deg,\partial)$, giving rise to its opposite complex $\frak{c}^{op}$ with boundary operator $\delta$ as above.  Write $C_*=C_*(\frak{c})$, $D_*=C_*(\frak{c}^{op})$, and use similar notation (\emph{e.g.}, $D_{k}^{(-\infty,\alpha)}=C_{k}^{(-\infty,\alpha)}(\frak{c}^{op})$) for the various graded and filtered 
subgroups.  Also we make the following assumption, which can be arranged to hold in the motivating applications by choosing the relevant covering appropriately:
\begin{assm}\label{inj} The period homomorphism $\omega\co \Gamma\to\mathbb{R}$ restricts to $\Gamma_0=\ker(deg)$ as an injection.  Consequently, the degree-zero part $\Lambda$ of the Novikov ring $\bar{\Lambda}$ is a field.
\end{assm}
(The reader can verify (or consult \cite[Theorem 4.1]{HS} to see) that the first sentence implies the second, given that we are taking the underlying coefficient ring $K$ of the Novikov ring to be a field.)

 We introduce here various pairings between $C_*$ and $D_*$.  

First, define \[ L\co D_*\times C_*\to \bar{\Lambda} \] by extending $\bar{\Lambda}$-bilinearly from the relation that $L(p,p)=1$ for each $p\in P$.  To be more explicit, if $S=\{s_1,\ldots,s_{|S|}\}$ and if we choose arbitrary $p_1,\ldots,p_{|S|}\in P$ so that $\pi(p_i)=s_i$, a general element of $C_*$ will have the form \[ c=\sum_{i=1}^{|S|}\sum_{g\in \Gamma}c_{i,g}g\cdot p_i \] where for each  each $M\in\mathbb{R}$ only finitely many $g$ have $\omega(g)<M$ and some $c_{i,g}\neq 0$.  Similarly (where all $\Gamma$-actions are the one in the definition of $\mathfrak{c}$, which is the inverse of the $\Gamma$-action in the definition of $\frak{c}^{op}$), a general element $d$  of $D_*$ will have the form \[ d=\sum_{i=1}^{|S|}\sum_{h\in \Gamma}d_{i,h}h^{-1}\cdot p_i \] where for each  each $M\in\mathbb{R}$ only finitely many $h$ have $\omega(h)<M$ and some $d_{i,h}\neq 0$. 

Then \[ L\left(\sum_{i=1}^{|S|}\sum_{h\in \Gamma}d_{i,h}h^{-1}\cdot p_i ,  \sum_{i=1}^{|S|}\sum_{g\in \Gamma}c_{i,g}g\cdot p_i \right)=\sum_{i=1}^{|S|}\sum_{g,h\in \Gamma}c_{i,g}d_{i,h}gh.\]

This definition is readily seen to be independent of the choice of $p_i$ with $\pi(p_i)=s_i$, and $L$ is clearly  $\bar{\Lambda}$-bilinear and nondegenerate.  

Further, observe that, for $p,q\in P$, we have \begin{align*} L(\delta p,q)&=L\left(\sum_{q'\in P}n(q',p)q',q\right)\\&=\sum_{g\in \Gamma}n(g^{-1}q,p)L(g^{-1}q,q)=\sum_{g\in \Gamma}n(g^{-1}q,p)g\\&=\sum_{g\in\Gamma}n(q,gp)g=L\left(p,\sum_{g\in\Gamma}n(q,gp)gp\right)=L(p,\partial q).\end{align*}

Note also that, for all $k$, $L$ restricts as a nondegenerate $\Lambda$-bilinear pairing \[ L\co D_{-k}\times C_{k}\to\Lambda.\]  Since $\Lambda$ is a field and since $C_k$ and $D_{-k}$ are finite-dimensional over $\Lambda$, it follows that $L$ sets up an identification $D_{-k}\cong Hom_{\Lambda}(C_k,\Lambda)$.  Moreover, the fact that $L(\delta p,q)=L(p,\partial q)$  implies that $L$ descends to a pairing $\underline{L}\co H_{-k}(D_*)\times H_k(C_*)\to \Lambda$ (where here and below we use the usual notation $H_j(A_*)$ for the degree-$j$ homology of a chain complex $A_*$), which, by the universal coefficient theorem, sets up an isomorphism $H_{-k}(D_*)\cong Hom_{\Lambda}(H_k(C_*),\Lambda)$.  In particular, $\underline{L}$ is nondegenerate. 

\begin{remark}In the case of the Novikov chain complex arising from a closed $1$-form $\theta$ on a closed manifold $M$, we have a cover $\tilde{M}\to M$, with deck transformation group $\Gamma_0$, to which $\theta$ pulls back as an exact form.  The singular chain complex $C_{*}^{sing}(\tilde{M})$ is then a $K[\Gamma_0]$-module where $K[\Gamma_0]$ is the group algebra, and a fundamental result in Novikov homology (due in this generality to Latour \cite{L94}) asserts that the Novikov chain complex $C_*(\frak{c})$ is chain homotopy equivalent to $C_{*}^{sing}(\tilde{M})\times_{K[\Gamma_0]}\Lambda$. In this case the pairings $L$ and $\underline{L}$ are essentially very classical objects; a standard construction dating back to Reidemeister \cite{R} gives a nondegenerate $K[\Gamma_0]$-bilinear pairing $C_{\dim M-*}^{sing}(\tilde{M})\times C_{*}^{sing}(\tilde{M})\to K[\Gamma_0]$, inducing a nondegenerate $K[\Gamma_0]$-bilinear pairing on homology.  It's not difficult to see that the opposite Novikov chain complex $D_{*}$ (which corresponds to Novikov homology for the $1$-form $-\theta$) has $D_{-*}$ chain equivalent to $C_{\dim M-*}^{sing}(\tilde{M})\otimes_{K[\Gamma_0]}\Lambda$ (with the ``opposite'' $K[\Gamma_0]$-action), and that, up to sign, the pairings $L$ and  $\underline{L}$ are induced from these classical Reidemeister pairings by coefficient extension from $K[\Gamma_0]$ to $\Lambda$.

\end{remark}

Now consider a different pairing, with values in $K$: \[ \Delta\co D_*\times C_*\to K\mbox{ defined by } \Delta\left(\sum_{p\in P}d_pp,\sum_{q\in P} c_qq\right)=\sum_{p\in P}d_pc_p.\]  Thus $\Delta$ is obtained by extending $K$-bilinearly from $\Delta(p,p)=1$ where $p\in P$.  Equivalently, where $\tau\co \bar{\Lambda}\to K$ is defined by $\tau\left(\sum a_g g=a_1\right)$ (where $1\in \Gamma$ is the identity), we have $\Delta=\tau\circ L$.  From this it follows that, like $L$, $\Delta$ obeys \[ \Delta(\delta d,c)=\Delta(d,\partial c).\]  (Indeed, for $p,q\in P$ we have $\Delta(\delta p,q)=\Delta(p,\partial q)=n(q,p)$.)

$\Delta\co D_*\times C_*\to K$ is obviously a nondegenerate $K$-bilinear pairing, which for each $k$ restricts to a nondegenerate pairing $D_{-k}\times C_k\to K$; however since $D_{-k}$ and $C_k$ have infinite dimension over $K$ this does not imply that $\Delta$ identifies $D_{-k}$ with $Hom_K(C_k,K)$, and in fact one can show that no such duality holds.  

Nonetheless, it remains true that the fact that $\Delta(\delta d,c)=\Delta(d,\partial c)$ implies that $\Delta$ induces a $K$-bilinear pairing \begin{equation}\label{deltabar}\underline{\Delta}\co H_{-k}(D_*)\times H_k(C_*)\to K.\end{equation}  Moreover, since the filtered complexes $D_{*}^{(-\infty,\alpha)}$ and $C_{*}^{(-\infty,\alpha]}$ are $K$-vector spaces (whereas they are not $\Lambda$-vector spaces), we can investigate the behavior of the $K$-valued pairing $\Delta$ on the filtered complexes.  In particular, observe that (since the action functional on the generators of the ``opposite complex'' $D_*$ is the negative of that on the generators of $C_*$), we have \begin{equation}\label{filtvanish} \Delta|_{D_{*}^{(-\infty,-\alpha)}\times C_{*}^{(-\infty,\alpha]}}=0.\end{equation}   
Consequently, writing \[ C_{*}^{(\alpha,\infty)}=\frac{C_*}{C_{*}^{(-\infty,\alpha]}},\]
we obtain a $K$-valued pairing \[ \Delta_{\alpha}\co D_{-k}^{(-\infty,-\alpha)}\times C_{k}^{(\alpha,\infty)}\to K.\]  Since $C_{*}^{(-\infty,\alpha]}$ is a subcomplex of $C_*$, the quotient $C_{*}^{(\alpha,\infty)}$ inherits a chain complex structure, and the relation $\Delta(\delta d,c)=\Delta(d,\partial c)$ again implies that the just-mentioned pairing descends to homology as a pairing \[ \underline{\Delta}_{\alpha}\co H_{-k}(D_{*}^{(-\infty,-\alpha)})\times H_k(C_{*}^{(\alpha,\infty)})\to K.\]  Here is our main  result:

\begin{theorem}\label{nondeg} Under Assumption \ref{inj},  for any $\alpha\in\mathbb{R}$ the homological pairing \[ \underline{\Delta}_{\alpha}\co H_{-k}(D_{*}^{(-\infty,-\alpha)})\times H_k(C_{*}^{(\alpha,\infty)})\to K\] is nondegenerate. 
\end{theorem} 


Again, since the relevant spaces are infinite-dimensional over $K$, this does not identify either of  them as the other's dual; however it does have useful applications as we will see below.

Note that it's not hard   to see that the pairing $\underline{\Delta}\co H_{-k}(D_*)\times H_k(C_*)\to K$ is nondegenerate.  Indeed, $\underline{\Delta}=\tau\circ\underline{L}$ where the Reidemeister pairing $\underline{L}\co H_{-k}(D_*)\times H_k(C_*)\to \Lambda$ is nondegenerate, and since $\Lambda$ is a field over which $H_{-k}(D_*)$ and $H_k(C_*)$ are vector spaces with $\underline{L}$ $\Lambda$-bilinear, it follows quickly that $\underline{\Delta}$ is nondegenerate.  

When $\omega(\Gamma_0)\leq \mathbb{R}$ is a discrete group, Theorem \ref{nondeg} is fairly straightforward, and a special case of it is a key point in the proofs of \cite[Lemma 2.2]{EP} and its generalization \cite[Lemma 5.1]{Ost} and hence in all the results that follow from these.  Note that, generally, the $K$-vector space $C_{k}^{(\alpha,\infty)}$ is a direct sum of one copy of $K$ for each $p\in P$ having $gr(p)=k$ and $\mathcal{A}(p)>\alpha$.  Thus, $Hom_K(C_{k}^{(\alpha,\infty)},K)$ is the direct product of these copies of $K$.  If $\omega(\Gamma_0)$ is discrete, it follows quickly from the definitions that $D_{-k}^{(-\infty,-\alpha)}$ is isomorphic to this direct product, with $\Delta_{\alpha}(\cdot,\cdot)$ acting as the standard pairing between $Hom_K(C_{k}^{(\alpha,\infty)},K)$ and $C_{k}^{(\alpha,\infty)}$. Thus when $\omega(\Gamma_0)$ is discrete Theorem \ref{nondeg} is a direct consequence of (the easy, field coefficient case of) the universal coefficient theorem for cohomology.  When $\omega(\Gamma_0)$ is dense in $\mathbb{R}$, however, the direct product $Hom_K(C_{k}^{(\alpha,\infty)},K)$ is much larger than $D_{-k}^{(-\infty,-\alpha)}$, so a different approach must be used. 

We briefly describe the strategy of the proof of Theorem \ref{nondeg} (or rather just the half of it which states that $\underline{\Delta}_{\alpha}$ is nondegenerate in its second argument, which is the part that we use in the  application to spectral invariants; the other half is a  bit different). The pairing $\Delta_{\alpha}$ is clearly nondegenerate on the relevant chain complexes: for a nonzero $c\in C_{k}^{(\alpha,\infty)}$ we can easily find $d\in D_{-k}^{(-\infty,-\alpha)}$ with $\Delta_{\alpha}(d,c)\neq 0$.  One then needs to show that if $c$ is a homologically nontrivial cycle then this $d$ can be chosen to have $\delta d=0$. The relation $\Delta_{\alpha}(\delta\cdot,\cdot)=\Delta_{\alpha}(\cdot,\partial \cdot)$ shows that  $\delta d=0$ provided that $\Delta(d,\partial c')=0$ for all $c'$.  To achieve this, we construct a map $\Pi\co C_k\to C_k$ which vanishes on $Im(\partial)$ but for which $\Pi(c)$ is homologous to $c$  and we show that $\Pi$ has an `adjoint' which maps the element $d$ with $\Delta(d,c)\neq 0$ to an element $\Pi^*d$ with $\Delta(\Pi^*d,c')=\Delta(d,\Pi c')$ for all $c'$.  The fact that $\Pi|_{Im(\partial)}=0$ shows that $\delta\Pi^*d=0$.  The main subtlety is in constructing   $\Pi$ in such a way that it behaves sufficiently well with respect to the filtration as to ensure that $\Pi^*d$ belongs to the subcomplex $D_{-k}^{(-\infty,-\alpha)}\subset D_{-k}$; this requires finding especially well-behaved bases for subspaces of the finite-dimensional vector space $\Lambda^N$.  Such bases can be constructed using the methods of \cite{Ucomp}; however a search of the literature reveals that (given that, unlike in \cite{Ucomp}, we work over a Novikov ring which is a field) the relevant ingredients can in fact be extracted by some simple arguments that date back to \cite{GG}, and this is the approach that we will present.

\subsection{Spectral numbers}

As is discussed for instance in \cite{Ucomp}, if $\frak{c}$ is a graded filtered Floer--Novikov complex and $a\in H_k(C_*(\frak{c}))$ we can associate to $a$ the \emph{spectral number} \[ \rho_{\frak{c}}(a)=\inf\{\ell(c)|c\in C_*(\frak{c}), [c]=a\},\] where $[c]$ denotes the homology class of $c$.  Our main purpose in proving Theorem \ref{nondeg} is to obtain the following corollary, generalizing \cite[Lemma 2.2]{EP}.

\begin{cor}\label{spectraldual} Let $\frak{c}$ be a graded filtered Floer-Novikov complex obeying Assumption \ref{inj}, let $\frak{c}^{op}$ be its opposite complex, and let $C_*$ and $D_*$ be the chain complexes associated to $\frak{c}$ and $\frak{c}^{op}$ respectively.  Then if $a\in H_k(C_*)\setminus\{0\}$, we have \[ \rho_{\frak{c}}(a)=-\inf\{\rho_{\frak{c}^{op}}(b)|b\in H_{-k}(D_*),\underline{\Delta}(b,a)\neq 0\},\] where $\underline{\Delta}\co H_{-k}(D_*)\times H_k(C_*)\to K$ is the pairing (\ref{deltabar}) induced by $\Delta$ on homology.
\end{cor}

\begin{proof} The argument is very similar to the proofs of \cite[Lemma 2.2]{EP} and \cite[Lemma 5.1]{Ost}. Suppose $\alpha<\rho_{\frak{c}}(a)$.  We have a short exact sequence of chain complexes \[ 0\to C_{*}^{(-\infty,\alpha]}\to C_*\to C_{*}^{(\alpha,\infty)}\to 0, \] inducing an exact sequence \[ \xymatrix{ H_{k}(C_{*}^{(-\infty,\alpha]})\ar[r]^{i_{\alpha}} & H_k(C_*) \ar[r]^{\pi_{\alpha}} & H_k(C_{*}^{(\alpha,\infty)}) } .\]  The fact that $\alpha<\rho_{\frak{c}}(a)$ means that $a$ is not represented by any chains of filtration level at most $\alpha$, so that $a\notin Im(i_{\alpha})$.  Hence $\pi_{\alpha}(a)\neq 0$.  So by the nondegeneracy of $\underline{\Delta}_{\alpha}$, there is $b_0\in H_{-k}(D_{*}^{(-\infty,-\alpha)})$ such that $\underline{\Delta}_{\alpha}(b_0,\pi_{\alpha}(a))\neq 0$.  Then where $d\in D_{*}^{(-\infty,-\alpha)}\subset D_*$ is any cycle representing $b_0$ in $H_{-k}(D_{*}^{(-\infty,-\alpha)})$ (and hence representing $b:=i_{-\alpha}b_0\in H_{-k}(D_*)$), we have $\Delta(d,c)=\underline{\Delta}_{\alpha}(b_0,\pi_{\alpha}(a))\neq 0$ whenever $c\in C_*$ is a representative of the class $a$.  Thus $\underline{\Delta}(b,a)\neq 0$, and since the representative $d$ of $b$ has filtration level less than $-\alpha$ we have $\rho_{\frak{c}^{op}}(b)<-\alpha$.  $\alpha$ was an arbitrary number smaller that $\rho_{\frak{c}}(a)$, so it follows that \[ \inf\{\rho_{\frak{c}^{op}}(b)|b\in H_{-k}(D_*),\underline{\Delta}(b,a)\neq 0\}\leq -\rho_{\frak{c}}(a).\]

The reverse inequality is easier, in that it follows directly from the definitions without depending on Theorem \ref{nondeg}.  Namely, suppose that $\alpha>\rho_{\frak{c}}(a)$; thus there must be some cycle $c\in C_{*}^{(-\infty,\alpha]}\subset C_*$ representing the homology class $a$.  If $b\in H_{-k}(D_*)$ is an arbitrary class satisfying $\underline{\Delta}(b,a)\neq 0$, then by the definition of $\underline{\Delta}$ it must hold that every representative $d\in D_*$ of the class $b$ satisfies $\Delta(d,c)\neq 0$.   
But then (\ref{filtvanish}) shows that this can only be true if no representative $d$ of $b$ belongs to $D_{*}^{(-\infty,-\alpha)}$, which amounts to the statement that $\rho_{\frak{c}^{op}}(b)\geq -\alpha$.  $b$ was an arbitrary class with $\underline{\Delta}(b,a)\neq 0$, while $\alpha$ was an arbitrary number exceeding $\rho_{\frak{c}}(a)$, and so we obtain that \[ 
\inf\{\rho_{\frak{c}^{op}}(b)|b\in H_{-k}(D_*),\underline{\Delta}(b,a)\neq 0\}\geq -\rho_{\frak{c}}(a),\] completing the proof.
\end{proof}

\subsection{Boundary depth}

In \cite{U09} we introduced the notion of ``boundary depth'' in Hamiltonian Floer theory.  The definition given there naturally adapts to the abstract setting that we consider here; indeed in the present context it perhaps more natural to refine the notion slightly to take gradings into account.  Thus, for a graded filtered Floer--Novikov complex $\frak{c}=(\pi\co P\to S,\mathcal{A},\omega,gr,deg,\partial)$ and for $k\in\mathbb{Z}$ we define \[ \beta_k(\frak{c})=\inf\left\{\beta\in [0,\infty) \left| (\forall \lambda\in\mathbb{R}) \left(\partial(C_{k+1})\cap C_{k}^{(-\infty,\gamma)}\subset \partial\left( C_{k+1}^{(-\infty,\gamma+\beta)} \right) \right) \right. \right\}.\]  So for $\beta>\beta_k(\frak{c})$ every chain $c$ in $C_k$ which is a boundary is in fact the boundary of some chain with filtration level at most $\ell(c)+\beta$.  As noted in \cite{U09}, \cite[Theorem 1.3]{Ucomp} implies the non-obvious fact that $\beta_k(\frak{c})$ is finite.  A variety of results from \cite{U09} demonstrate that the way in which the boundary depth of the Hamiltonian Floer complex changes as one varies the Hamiltonian yields interesting information in Hamiltonian dynamics.

\begin{remark}\label{nonzero} The definition implies that if the boundary operator $\partial \co C_{k+1}\to C_k$ identically vanishes, then $\beta_k(\frak{c})=0$.  On the other hand, we claim that if $\partial\co C_{k+1}\to C_k$ is nonzero then $\beta_k(\frak{c})>0$.  Indeed, with notation as in the definitions
in Section \ref{notation}, it's clear that \begin{equation}\label{geq} \beta_k(\frak{c})\geq \inf\{\mathcal{A}(p)-\mathcal{A}(q)|gr(p)=k+1, gr(q)=k,n(p,q)\neq 0\}.\end{equation} Now $gr^{-1}(\{k+1\})$ is the union of the orbits under $\Gamma_0$ of finitely many elements $p_1,\ldots,p_M$, and the relation $n(gp,gq)=n(p,q)$ implies that the right hand side is equal to \[ \inf\{\ell(p_i)-\ell(\partial p_i)|1\leq i\leq M, \partial p_i\neq 0\},\] which is an infimum over a (nonempty, since $\partial \neq 0$) finite set of positive numbers and so is positive.
\end{remark}

From Theorem \ref{nondeg}, we will now obtain a new formula for the boundary depth, which in particular demonstrates its symmetry under the operation of taking the opposite complex.  This formula is perhaps most appealingly expressed in terms of a natural ``linking form'' on the images of the boundary operators on the respective complexes.  If $x\in \delta(D_{-k+1})$ and $y\in \partial(C_{k})$, choose an arbitrary $c\in C_k$ with $\partial c=y$.  Now define \[ \lambda(x,y)=\Delta(x,c).\]  Of course a choice of $c$ with $\partial c=y$ was made here, but the value above is independent of the choice: if we take another $c'\in C_k$ with $\partial c'=y$, since $x\in Im(\delta)$ (say $x=\delta d$) we have \[ \Delta(x,c)-\Delta(x,c')=\Delta(\delta d,c-c')=\Delta(d,\partial(c-c'))=0.\]  Thus we have a well-defined map \[ \lambda\co \delta (D_{-k+1})\times \partial(C_{k})\to K.\]  Owing to the relation $\Delta(\delta d,c)=\Delta(d,\partial c)$, we could equally well have defined $\lambda(x,y)=\Delta(d,y)$ where we choose an arbitrary $d\in D_{-k+1}$ having $\delta d=x$.

\begin{cor} \label{bdepth} Let $\frak{c}$ be a graded filtered Floer-Novikov complex obeying Assumption \ref{inj}, let $\frak{c}^{op}$ be its opposite complex, and let $C_*$ and $D_*$ be the chain complexes associated to $\frak{c}$ and $\frak{c}^{op}$ respectively.  Then, for $k\in\mathbb{Z}$, \begin{equation}\label{bdeqn} \beta_{k-1}(\frak{c})=\beta_{-k}(\frak{c}^{op})=-\inf\left(\{0\}\cup\left\{\ell(y)+\ell^{op}(x)\left|x\in\delta(D_{-k+1}),y\in \partial(C_k),\lambda(x,y)\neq 0\right.\right\}\right).\end{equation}\end{cor} 

(Of course, $\ell^{op}$ denotes the filtration level in the opposite complex, \emph{i.e.}, for $x=\sum_{p\in P}x_pp\in D_{-k}$ we have $\ell^{op}(x)=\max\{\mathcal{A}^{op}(p)|x_p\neq 0\}=-\min\{\mathcal{A}(p)|x_p\neq 0\}.$)

\begin{proof} Due to the relation $\Delta(\delta d,c)=\Delta(d,\partial c)$, the symmetry of the expression on the right hand side of (\ref{bdeqn}) and the fact that $(\frak{c}^{op})^{op}=\frak{c}$ show that the second equality in (\ref{bdeqn}) will imply the first.  Accordingly we just prove the second equality.

We dispense with a trivial case first: if the operator $\partial\co C_k\to C_{k-1}$ vanishes, then so does $\delta\co D_{-k+1}\to D_{-k}$, and all the expressions in (\ref{bdeqn}) will be zero.  So assume hereinafter that $\partial\co C_k\to C_{k-1}$ is nonzero, which implies the same for $\delta\co D_{-k+1}\to D_{-k}$.  

For all nonzero $x\in \delta(D_{-k+1})\subset D_{-k}$ define \[ \beta(x)=\inf\{\alpha>\ell^{op}(x)| [x]=0\mbox{ in }H_{-k}(D_{*}^{(-\infty,\alpha)}) \}.\]  Now any $d\in D_{-k+1}$ with $\delta d=x$ must have $\ell^{op}(d)-\ell^{op}(x)$ at least equal to the right hand side of (\ref{geq}), which is independent of the choice of $d$ and was observed in Remark \ref{nonzero} to be positive.  Hence $\beta(x)>\ell^{op}(x)$.

Now let $\gamma$ be any real number with $\ell^{op}(x)<\gamma<\beta(x)$.  $x$ then represents a  nontrivial element in $D_{-k}^{(-\infty,\gamma)}$.  Theorem \ref{nondeg} hence finds a class $a\in H_{k}(C_{*}^{(-\gamma,\infty)})$ such that $\underline{\Delta}_{\alpha}([x],a)\neq 0$.  So if $c\in C_{k}$ is an arbitrary lift to $C_k$ of a representative of $a$, since $c$ projects to a relative cycle we will have $\ell(\partial c)\leq -\gamma$, and since that relative cycle represents $a$ we have $\Delta(x,c)\neq 0$.  Since $\gamma$ can be chosen arbitrarily close to $\beta(x)$, this shows that \begin{equation}\label{ineq1} -\beta(x)\geq\inf\{\ell(\partial c)|c\in C_k,\Delta(x,c)\neq 0\}.\end{equation}

On the other hand, suppose that $\alpha>\beta(x)$.  In this case there is an element $d\in D_{-k+1}^{(-\infty,\alpha)}$ with $\delta d=x$.  So if $c\in C_k$ has $\Delta(x,c)\neq 0$, the relations \[ \Delta(x,c)=\Delta(\delta d,c)=\Delta(d,\partial c)\]
together with (\ref{filtvanish}) show that $\ell(\partial c)>-\alpha$.  Since $\alpha$ can be taken arbitrarily close to $\beta(x)$, this and (\ref{ineq1}) together give \begin{equation}\label{betad} -\beta(x)=\inf\{\ell(\partial c)|c\in C_k,\Delta(x,c)\neq 0\}.\end{equation}

Now the definition of $\beta_{-k}(\frak{c}^{op})$ can be re-expressed as \[ 
\beta_{-k}(\frak{c}^{op})=\sup\{\beta(x)-\ell^{op}(x)|x\in \delta(D_{-k+1})\},\] which when combined with (\ref{betad}) establishes \[
\beta_{-k}(\frak{c}^{op})=\sup\{-\ell(\partial c)-\ell^{op}(x)|x\in \delta(D_{-k+1}),c\in C_k,\Delta(x,c)\neq 0\},\] which is equivalent to the statement of the corollary.
\end{proof}

\section{Proof of Theorem \ref{nondeg}}

\subsection{Subspaces of $\Lambda^N$}
Where $\omega\co\Gamma_0\to\mathbb{R}$ is an injective homomorphism defined on a finitely generated abelian group $\Gamma_0$, we consider a finite-dimensional vector space $\Lambda^N$ over the Novikov ring \[ \Lambda=\left\{\left.\sum_{g\in\Gamma_0}b_g g\right|b_g\in K,(\forall C\in\mathbb{R})(\#\{g\in \Gamma_0|b_g\neq 0,\omega(g)<C\}<\infty)\right\}. \]
We adopt some notation from \cite{Ucomp}. First, define a function $\nu\co \Lambda\to\mathbb{R}\cup\{\infty\}$ by \[ \nu\left(\sum_{g\in \Gamma_0}b_g g\right)=\min\{\omega(g)|b_g\neq 0\}\] (where the minimum of the empty set is  $\infty$), and define \[ \barnu\co\Lambda^N\to\mathbb{R}\cup\{\infty\}\mbox{ by }\barnu(v_1,\ldots,v_N)=\min\{\nu(v_i)|1\leq i\leq N\}.\]
 Write \[ \Lambda_{\geq 0}=\{v\in\Lambda|\nu(v)\geq 0\}, \quad \Lambda_+=\{v\in\Lambda|\nu(v)> 0\}.\]  Thus $\Lambda_{\geq 0}$ is a subring of the field $\Lambda$, and $\Lambda_+$ is an ideal in $\Lambda_{\geq 0}$, with $\Lambda_{\geq 0}/\Lambda_+\cong K$.  If $U\leq \Lambda^N$ is a $\Lambda$-vector subspace, set \[ U_{\geq 0}=U\cap (\Lambda_{\geq 0})^{N},\quad U_+=U\cap (\Lambda_+)^{N},\quad \tilde{U}=\frac{U_{\geq 0}}{U_+}.\]

Thus, for a $\Lambda$-vector subspace $U\leq \Lambda^N$, $U_{\geq 0}$ is a $\Lambda_{\geq 0}$-submodule of $\Lambda_{\geq 0}^{N}$, and $\tilde{U}$ is a $K$-vector subspace of $\widetilde{\Lambda^{N}}\cong K^N$.  If $u\in U_{\geq 0}$, denote by $\tilde{u}$ its equivalence class in $\tilde{U}=U_{\geq 0}/U_+$.

We adopt the following definition from \cite{GG} (where it is expressed in the general language of non-Archimedean normed vector spaces).

\begin{definition} An \emph{orthonormal basis} of a $\Lambda$-vector subspace $U$ of $\Lambda^N$ is a set $\{u_1,\ldots,u_m\}$ which is a basis for $U$ over $\Lambda$ and has the property that, for all $\lambda_1,\ldots,\lambda_m\in \Lambda$, \begin{equation}\label{on} \barnu\left(\sum_{i=1}^{m}\lambda_i u_i\right)=\min_{1\leq i\leq m}\nu(\lambda_i).\end{equation}
\end{definition}

Obviously, the standard basis $\{e_1,\ldots,e_N\}$ is an orthonormal basis for the whole space $\Lambda^N$.

\begin{lemma}{\cite[Hilfssatz 1]{GG}} \label{basis} Any subspace $U\leq \Lambda^N$ has an orthonormal basis.  Moreover, any orthonormal basis of a subspace $U\leq \Lambda^N$ extends to an orthonormal basis of $\Lambda^N$.
\end{lemma}

The proof in \cite{GG} is a bit terse, so we provide a proof for the benefit of the reader who may not be very conversant with non-Archimedean geometry.
\begin{proof}
Consider the $K$-vector subspace $\tilde{U}$ of $\widetilde{\Lambda^N}\cong K^N$.  Choose a basis for $\tilde{U}$; by the definition of $\tilde{U}$ this basis has the form $\{\tilde{u}_1,\ldots,\tilde{u}_m\}$ for some $u_1,\ldots,u_m\in U_{\geq 0}$.  Consider an element $\sum_{i=1}^{m}\lambda_iu_i$ for some $\lambda_1,\ldots,\lambda_m\in\Lambda$ which are not all zero. There is then $g\in \Gamma_0$ with $\omega(g)=-\min\{\nu(\lambda_1),\ldots,\nu(\lambda_m)\}$, and we can write \[ g\lambda_i=k_i+\lambda'_{i}\] with $k_i\in K$, $\lambda'_i\in \Lambda_+$, and at least one $k_i\neq 0$.  Then \[ \widetilde{\left(\sum_{i=1}^{m}g\lambda_i u_i\right)}=\sum_{i=1}^{m}k_i\tilde{u}_i\neq 0\] by the linear independence of the $\tilde{u}_i$, and hence $\barnu\left( \sum_{i=1}^{m}g\lambda_i u_i\right)=0$, so \[ \barnu\left(\sum_{i=1}^{m}\lambda_iu_i\right)=-\omega(g)=\min_{1\leq i\leq m}\nu(\lambda_i).\]  It follows that the $u_i$ are linearly independent, and moreover form an orthonormal basis for the subspace $U'$ of $\Lambda^N$ which they span (we'll later conclude that $U'=U$). 

Now complete the linearly independent set $\{\tilde{u}_1,\ldots\tilde{u}_m\}$ to a basis \[ 
\{\tilde{u}_1,\ldots\tilde{u}_m, \tilde{v}_1,\ldots,\tilde{v}_{N-m}\}\] for $\widetilde{\Lambda^N}\cong K^N$.  (Here each $v_i\in \Lambda^N$). The argument of the previous paragraph shows that $\{u_1,\ldots,u_m, v_1,\ldots,v_{N-m}\}$ is an orthonormal basis for the subspace of $\Lambda^N$ which it spans, but by counting dimensions we see that this subspace is all of $\Lambda^N$.  Let $V=span_{\Lambda}\{v_1,\ldots,v_{N-m}\}$, so we have $U'\oplus V=\Lambda^N$.  Since $U'\leq U$, to show that $U=U'$ it suffices to show that $U\cap V=\{0\}$.    If on the contrary $U\cap V\neq\{0\}$, there would be $\lambda_i\in\Lambda$, not all zero, with $v:=\sum_{i=1}^{N-m}\lambda_iv_i\in U$.  Then where $g\in \Gamma_0$ is chosen so that $\omega(g)=-\min_i \nu(\lambda_i)$, $gv$ would descend to a nontrivial element in $\tilde{U}$ which is a linear combination of the $\tilde{v}_i$, a contradiction with the fact that the $\tilde{u}_j$ and $\tilde{v}_i$ are linearly independent.  This proves that $U\cap V=\{0\}$, hence that $U'=U$ and so $\{u_1,\ldots,u_m\}$ is an orthonormal basis for $U$.

To prove the second sentence of the lemma, let $\{u_1,\ldots,u_m\}$ be an orthonormal basis for $U$. Setting $\lambda_i=\delta_{ij}$ in (\ref{on}) shows that each $\barnu(u_j)=0$, so that we have well-defined elements $\tilde{u}_j\in \tilde{U}$.  If $k_1,\ldots,k_m\in K$ are not all zero, (\ref{on}) shows that $\barnu(\sum_{i=1}^{m}k_iu_i)=0$ and hence that $\sum_{i=1}^{m}k_i\tilde{u}_i\neq 0$. Thus $\{\tilde{u}_1,\ldots,\tilde{u}_m\}$ is linearly independent in $\widetilde{\Lambda^N}$, and hence extends to a basis $\{\tilde{u}_1,\ldots,\tilde{u}_N\}$ 
for $\widetilde{\Lambda^N}\cong K^N$.  Just as in the previous paragraph, $\{u_1,\ldots,u_N\}$ is then an orthonormal basis for the subspace of $\Lambda^N$ which it spans, which is all of $\Lambda^N$ by a dimension count.
\end{proof}

\begin{cor} \label{piprime} Let $U\leq \Lambda^N$ be a $\Lambda$-vector subspace.  Then there is a $\Lambda$-linear map $\Pi'\co \Lambda^N\to\Lambda^N$ such that \begin{itemize} \item[(i)] $U=\ker\Pi'$; 

\item[(ii)] For all $w\in \Lambda^N$ we have \[ \barnu(\Pi'w)\geq \barnu(w);\]
\item[(iii)]$\Pi'^2=\Pi'$.
\end{itemize}
\end{cor}

\begin{proof} By Lemma \ref{basis}, let $\{u_1,\ldots,u_m\}$ be an orthonormal basis for $U$, and extend this to an orthonormal basis $\{u_1,\ldots,u_N\}$ for $\Lambda^N$.  Define $\Pi'\co \Lambda^N\to\Lambda^N$ to be the projection onto $span_{\Lambda}\{u_{m+1},\ldots u_{N}\}$, \emph{i.e.}, $\Pi'\left(\sum_{i=1}^{N}\lambda_iu_i\right)=\sum_{i=m+1}^{N}\lambda_iu_i$.  Then if $w=\sum_{i=1}^{N}\lambda_iu_i$, by the definition of orthonormality we have \[
\barnu(\Pi'w)=\min_{m+1\leq i\leq N}\nu(\lambda_i)\geq \min_{1\leq i\leq N}\nu(\lambda_i)=\barnu(w).\]\end{proof}

\subsection{Adjoints}  We return to the general abstract setting of a graded filtered Floer-Novikov complex $\frak{c}$ and its opposite complex $\frak{c}^{op}$, with associated Floer chain complexes $C_*$ and $D_*$ respectively.  The following proposition in essence shows that, while   $\Delta$ identifies $D_{-k}$ with just a subspace of the dual space of $C_{k}$, this subspace is preserved by the adjoint of any $\Lambda$-linear endomorphism of $C_k$.  Furthermore, as will be important for what follows, the adjoint operation behaves well with respect to the filtrations.

\begin{prop} \label{adjoint} Let $A\co C_k\to C_k$ be a $\Lambda$-linear map.  Then there is a $\Lambda$-linear map $A^*\co D_{-k}\to D_{-k}$ with the property that, for all $d\in D_{-k}, c\in C_k$ we have \[ \Delta(A^*d,c)=\Delta(d,Ac).\]  Moreover, if for some $\ep>0$ $A$ has the property that $A(C_{k}^{(-\infty,\beta)})\subset C_{k}^{(-\infty,\beta+\ep)}$ for all $\beta\in\mathbb{R}$, then $A^*$ has the property that $A^*(D_{-k}^{(-\infty,\beta)})\subset D_{-k}^{(-\infty,\beta+\ep)}$ for all $\beta\in\mathbb{R}$.

\end{prop}

\begin{proof} 
Write $P_k=gr^{-1}(\{k\})\subset P$.  The map $A\co C_{k}\to C_{k}$ is determined by the values $m(q,q')\in K$ such that, for $q\in P_k$ we have \[ Aq=\sum_{q'\in P_k}m(q,q')q'.\]  These $m(q,q')$ are constrained by the restrictions that $m(gq,gq')=m(q,q')$ for $g\in \Gamma_0$; that for any $q$ and any $C\in\mathbb{R}$ there be just finitely many $q'$ with $\mathcal{A}(q')>C$ and $m(q,q')\neq 0$; and  (under the restriction of the last sentence of the proposition) that $m(q,q')=0$ if $\mathcal{A}(q')>\mathcal{A}(q)+\ep$. 

Then for $p,q\in P_k$ \begin{equation}\label{adj} \Delta(p,Aq)= \Delta\left(p,\sum_{q'\in P_k}m(q,q')q'\right)=m(q,p).\end{equation}  Given $p\in P_k$, consider the expression $\sum_{q\in P_k}m(q,p)q$.  Let $C\in\mathbb{R}$ and suppose, for contradiction, that there were infinitely many $q\in P_k$ with $m(q,p)\neq 0$ and $\mathcal{A}^{op}(q)>C$ (\emph{i.e.}, $\mathcal{A}(q)<-C$).  Since $P_k$ consists of just finitely many orbits of the $\Gamma_0$-action, there would then be some $q_0\in P_k$ and infinitely many different $g\in \Gamma_0$ such that $m(g\cdot q_0,p)\neq 0$ and $\mathcal{A}(q_0)-\omega(g)<-C$.  But $m(g\cdot q_0,p)=m(q_0,g^{-1}\cdot p)$, and for all of these $g$ we would have \[ \mathcal{A}(g^{-1}\cdot p)=\mathcal{A}(p)+\omega(g)>\mathcal{A}(p)+\mathcal{A}(q_0)+C,\] which would contradict the finiteness condition on the $m(q_0,\cdot)$.  

This proves that the expression $\sum_{q\in P_k}m(q,p)q$ validly defines an element of $D_{-k}$, which we denote by $A^*p$.  The fact that $m(gq,gp)=m(q,p)$ shows that $A^*$ commutes with the action of $\Gamma_0$ on $P_k$.  Hence $A^*$ extends to a $\Lambda$-linear map $A^*\co D_{-k}\to D_{-k}$.  By (\ref{adj}), we have, for $p,q\in P_k$, \[ \Delta(p,Aq)=m(q,p)=\Delta(A^*p,q),\] which by the $K$-bilinearity of $\Delta$ is easily seen to imply that $\Delta(A^*d,c)=\Delta(d,Ac)$ for all $c\in C_k$, $d\in D_{-k}$.

For the statement about the filtrations\footnote{Incidentally, the above was unnecessary to prove the mere existence of the adjoint, since we have $\Delta=\tau\circ L$ where $L$ makes $D_{-k}$ dual over $\Lambda$ to $C_k$.  However, the last part of the proposition depends on the explicit form for $A^*$, which is why we gave these details.}, note that if $d=\sum_{p\in P_k}d_pp\in D_{-k}^{(-\infty,\beta)}$, then $A^*d=\sum_{q\in P_k}y_qq$ where every $q$ with $y_q\neq 0$ has the property that for some $p$ with $\mathcal{A}^{op}(p)<\beta$ we have $m(q,p)\neq 0$.  By the assumption on the behavior of $A$ with respect to the filtration on $C_{k}$, we have $\mathcal{A}(p)-\mathcal{A}(q)\leq \ep$ whenever $m(q,p)\neq 0$.  So since $\mathcal{A}^{op}=-\mathcal{A}$ we obtain, when $m(q,p)\neq 0$, $\mathcal{A}^{op}(q)\leq \ep+\mathcal{A}^{op}(p)<\beta+\ep$.  This proves that $A^*d\in D_{-k}^{(-\infty,\beta+\ep)}$, completing the proof of the proposition.
\end{proof}

\subsection{Nondegeneracy on the right}  We now begin in earnest the proof of the main result, Theorem \ref{nondeg}.  As explained in Section \ref{pairings}, in the special case that $\omega(\Gamma_0)$ is discrete, $\Delta_{\alpha}$ identifies $D_{-k}^{(-\infty,-\alpha)}$ with the dual over $K$ of $C_k^{(\alpha,\infty)}$ (as was exploited in the context of Hamiltonian Floer theory in \cite{EP},\cite{Ost}).  Consequently if $\omega(\Gamma_0)$ is discrete then Theorem \ref{nondeg} simply follows from the universal coefficient theorem.  Accordingly we assume for the rest of the paper that $\omega(\Gamma_0)$ is non-discrete, and hence dense in $\mathbb{R}$.
 
Write $P_k=gr^{-1}(\{k\})$ and $S_k=\pi(P_k)$; thus $\pi|_{P_k}$ defines a principal $\Gamma_0$-bundle over $S_k$.  If $S_k=\{s_1,\ldots,s_N\}$ and $\vec{p}=(p_1,\ldots,p_N)$ is an arbitrary tuple of points of $P_k$ with $\pi(p_i)=s_i$, we have an isomorphism of $\Lambda$-vector spaces \[ \Phi_{\vec{p}}\co C_k\to \Lambda^N \] arising from the fact that $C_k$ can be written as \[ C_k=\oplus_{i=1}^{N}\Lambda\cdot p_i.\]
Setting $t_i=\mathcal{A}(p_i)$ and $\vec{t}(\vec{p})=(t_1,\ldots,t_N)$, one evidently has, for any $c\in C_k(\frak{c})$, \[ \ell(c)=-\barnu_{\vec{t}(\vec{p})}(\Phi_{\vec{p}}(c)).\]  Here the notation $\barnu_{\vec{t}(\vec{p})}$ is borrowed from \cite{Ucomp}: for $\vec{t}=(t_1,\ldots,t_N)$ and $v=(v_1,\ldots,v_N)\in \Lambda^N$ we define \[ \barnu_{\vec{t}}=\min\{\nu(v_i)-t_i|1\leq i\leq N\}.\]
In light of this, \cite[Theorem 2.5]{Ucomp} (applied to the matrix representing the $\Lambda$-linear map $\partial\co C_{k+1}\to C_k$) immediately shows:
\begin{prop}{\cite[Theorem 2.5]{Ucomp}}\label{tight} If $c_0\in C_k$, there is $b_0\in C_{k+1}$ such that, where $c=c_0-\partial b_0$, we have \[ \ell(c)\leq \ell(c_0-\partial b)\] for all $b\in C_{k+1}$.\end{prop}

Now let $\alpha\in \mathbb{R}$ and let $a$ be an arbitrary nonzero class in $H_k(C_{*}^{(\alpha,\infty)})$.  Choose an arbitrary $c_0\in C_k$ so that the projection of $c_0$ to  $C_{k}^{(\alpha,\infty)}$ represents the class $a$.  Let $c$ be as in Proposition \ref{tight}; since $c$ differs from $c_0$ by a boundary the projection of $c$ to $C_{k}^{(\alpha,\infty)}$ also represents $a$.  In particular, since $a\neq 0$, we have $c\notin C_{k}^{(-\infty,\alpha]}$; thus $\ell(c)>\alpha$.  Choose a number $\ep$ with \[ 0<\ep<\frac{1}{2}(\ell(c)-\alpha).\]  Because $\omega(\Gamma_0)$ is dense in $\mathbb{R}$, we can choose $p_1,\ldots,p_N\in P_k$ as at the start of this subsection with \[ \alpha\leq \mathcal{A}(p_i)<\alpha+\ep.\]  The associated isomorphism $\Phi_{\vec{p}}\co C_{k}\to  \Lambda^N$ has the property that that, if $c\in C_k$, we have \[ \barnu(\Phi_{\vec{p}}(c))+\ell(c)\in \left[\min_{1\leq i\leq N}\mathcal{A}(p_i),\max_{1\leq i\leq N}\mathcal{A}(p_i)\right]\subset [\alpha,\alpha+\ep).\]  Consequently,  for $c_1,c_2\in C_{k}$ we have \begin{equation}\label{skew} \left|\left(\barnu(\Phi_{\vec{p}}(c_1))-\barnu(\Phi_{\vec{p}}(c_2))\right)-\left(\ell(c_2)-\ell(c_1)\right)\right|<\ep\end{equation}

Consider the subspace \[ U=\Phi_{\vec{p}}\left(Im(\partial\co C_{k+1}\to C_k)\right)\leq \Lambda^N\] and let $\Pi'$ be as in Corollary \ref{piprime}.
Define $c'\in C_k$ by \[ \Phi_{\vec{p}}(c')=\Pi'\Phi_{\vec{p}}(c).\]  Since $\Pi'(\Pi'-1)=0$, we have $\Phi_{\vec{p}}(c')-\Phi_{\vec{p}}(c)\in U$, so that $c'-c\in Im(\partial)$; thus $c'$ is a representative of our arbitrary nonzero $a\in H_k(C_{*}^{(\alpha,\infty)})$.



 It then follows from  (\ref{skew}) that the map \[ \Pi\co C_k\to C_k \mbox{ defined by } \Pi=\Phi_{\vec{p}}^{-1}\circ \Pi'\circ\Phi_{\vec{p}} \] is a $\Lambda$-linear map which obeys

\begin{itemize} 
\item[(i)] $\Pi|_{Im(\partial)}=0$ 
\item[(ii)] $\Pi(c')=c'$, and 
\item[(iii)] For any $\beta\in\mathbb{R}$, $\Pi(C_{k}^{(-\infty,\beta)})\subset  C_{k}^{(-\infty,\beta+\ep)}$.
\end{itemize}

Note that, by our choice of $c$, since $c'-c\in Im(\partial)$ we have \[ \ell(c')\geq \ell(c)>\alpha+2\ep.\]

Writing \[ c'=\sum_{q\in P_k}c'_q q,\] we can choose $p\in P_k$ with $c'_p\neq 0$ and $\mathcal{A}(p)=\ell(c')$.  Considered as a generator for the opposite complex $\frak{c}^{op}$, $p$ has $\mathcal{A}^{op}(p)<-\alpha-2\ep$.  Thus $p$ can be viewed as an element of $D_{-k}^{(-\infty,-\alpha-2\ep)}$.   Evidently $\Delta(p,c')=c'_p\neq 0$.

Now apply Proposition \ref{adjoint} to our projection $\Pi$, yielding a map $\Pi^*\co D_{-k}\to D_{-k}$ which raises filtration by no more than $\ep$ and obeys $\Delta(\Pi^*d,c)=\Delta(d,\Pi c)$ for all $d\in D_{-k},c\in C_k$.  Set $d_0=\Pi^*p$.  
We then have \begin{itemize} \item[(i)] $d_0=\Pi^*p\in D_{-k}^{(-\infty,-\alpha-\ep)}\subset D_{-k}^{(-\infty,-\alpha)}$;
\item[(ii)] For all $c''\in C_{k+1}$, $\Delta(\delta d_0,c'')=\Delta(\Pi^*p,\partial c'')=\Delta(p,\Pi(\partial c''))=0$, and hence $\delta d_0=0$.
\item[(iii)] $\Delta(d_0,c')=\Delta(p,\Pi c')=\Delta(p,c')=c'_p\neq 0$.\end{itemize}

By (i) and (ii) above, $d_0$ represents a class $[d_0]\in H_k(D_{-k}^{(-\infty,-\alpha)})$ and then since (the projection of)  $c'$ represents our arbitrary nonzero $a\in H_k(C_{*}^{(\alpha,\infty)})$ (iii) above shows that $\underline{\Delta}_{\alpha}([d_0],a)\neq 0$.  So we have proven:

\begin{prop} For any nonzero class $a\in H_k(C_{*}^{(\alpha,\infty)})$ there is $b\in H_{-k}(D_{*}^{(-\infty,-\alpha)})$ such that $\underline{\Delta}_{\alpha}(b,a)\neq 0$.\end{prop}

\subsection{Nondegeneracy on the left}
The above proves part of Theorem \ref{nondeg} (indeed, it proves the part that is used in the main application Corollary \ref{spectraldual}); to complete the proof of Theorem \ref{nondeg} we must show that, likewise, if $0\neq b\in H_{-k}(D_{*}^{(-\infty,-\alpha)})$ then there is $a\in H_k(C_{*}^{(\alpha,\infty)})$ with $\underline{\Delta}_{\alpha}(b,a)\neq 0$; this half of the theorem is the one that is used in the proof of Corollary \ref{bdepth}.

Here there are two distinct cases, depending on whether or not the image of $b$ under the inclusion-induced map $i_{-\alpha}\co H_{-k}(D_{*}^{(-\infty,-\alpha)})\to H_{-k}(D_*)$ is trivial.  The more straightforward case is that in which $i_{-\alpha}b\neq 0$.  Let $d_0\in D_{-k}^{(-\infty,-\alpha)}\subset D_{-k}$ be a representative of $b$.  Choose  $p_1,\ldots,p_N\in P_k$ as before, thus establishing an isomorphism  
$\Phi_{\vec{p}}\co D_{-k}\to \Lambda^N$.    As before, Corollary \ref{piprime} produces a $\Lambda$-linear projection $\Pi\co D_{-k}\to D_{-k}$ with $\ker\Pi=Im(\delta)$; set $d=\Pi d_0$. So since $\Pi^2=\Pi$ we have $d-d_0\in Im(\delta)$. Then Proposition \ref{adjoint} (and the fact that $(\frak{c}^{op})^{op}=\frak{c}$) provides a $\Lambda$-linear map $\Pi^*\co C_k\to C_k$ such that, for all $d'\in D_{-k},c'\in C_k$, we have $\Delta(\Pi d',c')=\Delta(d',\Pi^*c')$.  If $p$ is any element of $P_k$ appearing with nonzero coefficient $d_p$ in $d$, we have $\partial(\Pi^*p)=0$ (since $\Pi\circ\delta=0$) and 
\[ \Delta(d_0,\Pi^*p)=\Delta(\Pi d_0,p)=\Delta(d,p)=d_p\neq 0.\]  Hence the homology class $a$ of the projection of $\Pi^*p$ to $C^{(\alpha,\infty)}=C_*/C_{*}^{(-\infty,\alpha]}$ satisfies $\underline{\Delta}_{\alpha}(b,a)\neq 0$.

A somewhat different approach is needed for the case that the given nonzero class $b\in H_{-k}(D_{*}^{(-\infty,-\alpha)})$ has $i_{-\alpha}b=0$. 
In the interests of brevity, we will give full details of the argument just in the case that \begin{equation}\label{assumealpha} -\alpha\notin \mathcal{A}^{op}(P_{k-1});\end{equation} after we give a full treatment of this case we will sketch an argument that works when this assumption fails. (Of course, $\mathcal{A}^{op}(P_{k-1})$ is countable, so in applications such as Corollary \ref{bdepth} one usually only needs the result for $\alpha$ satisfying (\ref{assumealpha}).)

 Let $d\in D_{-k}^{(-\infty,-\alpha)}$ be any representative of $b$. 
 Thus $d\in Im(\delta\co D_{-k+1}\to D_{-k})$, but $d\notin Im(\delta|_{D_{-k+1}^{(-\infty,-\alpha)}})$. Choose $x\in D_{-k+1}$ with $\delta x=d$.  In light of the assumption (\ref{assumealpha}) we must have $\ell^{op}(x)>-\alpha$.  Let $0<\ep<\frac{1}{2}(\ell^{op}(x)-(-\alpha))$.   Choose elements $p_1,\ldots,p_M\in gr^{-1}(\{k-1\})=(gr^{op})^{-1}(\{-k+1\})$ from each of the sets $\pi^{-1}\{s_i\}$ for $s_i\in \pi(gr^{-1}(\{k-1\}))$, so that $D_{-k+1}= \oplus_{i=1}^{M}\Lambda\cdot^{op} p_i$, and let $\Phi_{\vec{p}}\co D_{-k+1}\to\Lambda^M$ be the associated isomorphism.  Since $\omega(\Gamma_0)$ is dense we can and do assume that \begin{equation}\label{aalpha}-\alpha< \mathcal{A}^{op}(p_i)<-\alpha+\ep \end{equation} for all $i$.  Choose $d_2,\ldots,d_m\in Im(\delta)\subset D_{-k}$ so that $\{d,d_2,\ldots,d_m\}$ spans $Im(\delta)$, and let $x_i$ be such that $\delta x_i=d_i$.  Choose the unique $g\in \Gamma_0$ so that $\barnu(\Phi_{\vec{p}}(g\cdot^{op} x))=0$.  Consequently $\{\Phi_{\vec{p}}(g\cdot^{op} x)\}$ defines an orthonormal basis for the one-dimensional subspace of $\Lambda^M$ which it spans.  The argument in the proof of Lemma \ref{basis} extends this basis to an orthonormal basis $\{\Phi_{\vec{p}}(g\cdot^{op} x),\Phi_{\vec{p}}(x'_2),\ldots,\Phi_{\vec{p}}(x'_m)\}$ of the subspace $span_{\Lambda}\{\Phi_{\vec{p}}(x),\Phi_{\vec{p}}(x_2),\ldots,\Phi_{\vec{p}}(x_m)\}\leq \Lambda^M$.  Let $d'_i=\delta x'_i$; thus $\{d,d'_2,\ldots,d'_m\}$ is a basis for $Im(\delta)$.  Extend this basis arbitrarily to a basis $\{d,d'_2,\ldots,d'_N\}$ for $D_{-k}$ over $\Lambda$, and let $\Pi\co D_{-k}\to D_{-k}$ denote the projection onto the first coordinate, \emph{i.e.},
\[ \Pi\left(\lambda\cdot^{op}  d+\sum_{i=2}^{N}\lambda_i\cdot^{op} d'_i\right)=\lambda\cdot^{op} d.\]  If $\lambda=\sum_{h\in \Gamma_0}a_h h$ is an arbitrary element of $\Lambda$, set $\underline{\lambda}=\sum_{h\in \Gamma_0,\omega(h)\leq 0}a_h h$.  Now define $\phi\co span_{\Lambda}\{d\}\to span_{\Lambda}\{d\}$ by $\phi(\lambda\cdot^{op}  d)=\underline{\lambda}\cdot^{op} d$, and set \[ \underline{\Pi}=\phi\circ \Pi;\] thus $\underline{\Pi}$ is a $K$-linear map (not a $\Lambda$-linear map) with $\underline{\Pi}(d)=d$.  

We claim that $\underline{\Pi}$ vanishes on $V_{\alpha}:=Im(\delta|_{D_{-k+1}^{(-\infty,-\alpha)}})$. To see this, note first that the constraint on the values of $\mathcal{A}^{op}(p_i)$ implies that $\Phi_{\vec{p}}(D_{-k+1}^{(-\infty,-\alpha)})\subset \Lambda_{\geq 0}^{N}$.  This together with the fact that 
$\{\Phi_{\vec{p}}(g\cdot^{op} x),\Phi_{\vec{p}}(x'_2),\ldots,\Phi_{\vec{p}}(x'_m)\}$ is an orthonormal basis for the space that it spans implies that any element of $V_{\alpha}$ can be written uniquely as \[ v=\delta(\lambda g\cdot^{op} x+\sum_{i=2}^{m}\lambda_i\cdot^{op} x'_i)= \lambda g\cdot^{op} d+\sum_{i=2}^{m}\lambda_i\cdot^{op} d'_i\] with each $\lambda_i\in \Lambda_{\geq 0}$ and $\lambda\in \Lambda_{\geq 0}$. Moreover, we have \[ \barnu\left(\lambda \Phi_{\vec{p}}(g\cdot^{op} x)+\sum_{i=2}^{m}\lambda_i\Phi_{\vec{p}}(x'_i)\right)=\min\{\nu(\lambda),\nu(\lambda_i)\}.\] Now the constraint on the $\mathcal{A}^{op}(p_i)$  implies that the various $\ell^{op}(g\cdot^{op} x)$ and $\ell^{op}(x'_i)$ differ from each other by less than $\ep$, so it follows that \begin{equation}\label{ellep} \ell^{op}(\lambda g\cdot^{op} x+\sum_{i=2}^{m}\lambda_i\cdot^{op} x'_i)>\ell^{op}(\lambda g\cdot^{op} x)-\ep.\end{equation}    Together, these facts imply that if $\ell^{op}(\lambda g\cdot^{op} x+\sum_{i=2}^{m}\lambda_i\cdot^{op} x'_i)<\alpha$, then $\ell^{op}(\lambda g\cdot^{op} x)<-\alpha+\ep$, and so since $\ell^{op}(x)>-\alpha+2\ep$ we must have $\nu(\lambda g)>\ep$.  But in this case \[ \underline{\Pi}\left(\delta\left(\lambda g\cdot^{op} x+\sum_{i=2}^{m}\lambda_i\cdot^{op} x'_i\right)\right)=\phi(\delta(\lambda g\cdot^{op} x))=\phi(\lambda g\cdot^{op} d)=0.\]
Thus $\underline{\Pi}|_{V_{\alpha}}=0$, as claimed.

The intention now is to let $p\in P_k$ appear with nonzero coefficient $d_p$ in $d$, and then, viewing $p$ as an element of $C_k$, define $c\in C_k$ by the relation \begin{equation} \label{definec}\Delta( y,c)=\Delta(\underline{\Pi} y,p).\end{equation}  If this relation does indeed validly define an element $c\in C_k$, then we'll be done.  Indeed, this would imply that $\Delta(d,c)=\Delta(d,p)=d_p\neq 0$, while for $x'\in D_{-k+1}^{(-\infty,-\alpha)}$ we would have $\Delta(x', \partial c)=\Delta(\delta x',c)=\Delta (\underline{\Pi} \delta x',p)=0$, so that $c$ would have $\partial c\in C_{k-1}^{(-\infty,\alpha]}$ and so $c$ would determine a class in $H_k(C_{*}^{(\alpha,\infty)})$ having nonzero pairing with $[d]=b$.  

To verify that (\ref{definec})  validly defines an element $c$, choose $q_1,\ldots,q_N\in P_k$ so that the $\pi(q_i)$ are the distinct elements of $S_k$, and thus $D_{-k}=\oplus_{i=1}^{N}\Lambda\cdot^{op} q_i$; we take $q_1$ equal to the element $p$ of the previous paragraph.   
 For some $\mu_i=\sum_{g\in \Gamma_0}\mu_{i,g}g\in \Lambda$ we have $\Pi q_i=\mu_i\cdot^{op}  d$.  Hence for any $h\in \Gamma_0$, \[ \underline{\Pi}(h\cdot^{op} q_i)=\phi\left(\sum_{g\in \Gamma_0}\mu_{i,g}gh\cdot^{op} d\right)=\phi\left(\sum_{g\in \Gamma_0}\mu_{i,gh^{-1}}g\cdot^{op} d\right)=\sum_{g\in\Gamma_0,\omega(g)\leq 0}\mu_{i,gh^{-1}}g\cdot^{op} d.\]  Thus, writing \[ d=\sum_{i=1}^{N}\sum_{k\in \Gamma_0}d_{i,k}k\cdot^{op} q_i,\] we have \begin{align*} \Delta\left(\underline{\Pi}h\cdot^{op} q_i,q_1\right)&=\Delta\left(\sum_{k,g\in \Gamma_0,\omega(g)\leq 0}d_{1,k}\mu_{i,gh^{-1}}kg\cdot^{op} q_1,q_1\right)\\&=\sum_{g\in\Gamma_0,\omega(g)\leq 0}d_{1,g^{-1}}\mu_{i,gh^{-1}}.\end{align*}

Write $c_{i,h}=\sum_{g\in\Gamma_0,\omega(g)\leq 0}d_{1,g^{-1}}\mu_{i,gh^{-1}}$.  (Since the $\mu_i$ belong to $\Lambda$, for any given $h$ there will be just finitely many $g$ with $\omega(g)\leq 0$ and $\mu_{i,gh^{-1}}\neq 0$, so the sum defining $c_{i,h}$ has just finitely many terms).  
Substituting $k=gh^{-1}$ we can rewrite \[ c_{i,h}=\sum_{k\in\Gamma_0,\omega(k)\leq -\omega(h)}d_{1,h^{-1}k^{-1}}\mu_{i,k}.\]  Now if $m\in\mathbb{R}$, there are just finitely many $k$, say $k_1,\ldots,k_r$, such that $\omega(k)\leq -m$ and $\mu_{i,k}\neq 0$ for some $i$.  So if $h$ is to have $c_{i,h}\neq 0$ and $\omega(h)\geq m$, it must hold that $d_{1,h^{-1}k_{j}^{-1}}\neq 0$ for some $j\in\{1,\ldots,r\}$.  The finiteness condition on $\sum d_{i,g}g\in \Lambda$ shows that there are just finitely many $h$ with $\omega(h)\geq m$ and $d_{h^{-1}k_{j}^{-1}}\neq 0$.  This proves that \[ \sum_{h\in\Gamma_0}c_{i,h}h^{-1}\in\Lambda.\] 
Hence, bearing in mind the fact that the $\Gamma$-action $\cdot^{op}$ in $\frak{c}^{op}$ is opposite to the action $\cdot$ in $\frak{c}$, the expression \[ c=\sum_{i=1}^{N}\sum_{h\in\Gamma_0}c_{i,h}h\cdot^{op}q_i=\sum_{i=1}^{N}\sum_{h\in\Gamma_0}c_{i,h}h^{-1}\cdot q_i\] defines an element $c\in C_k$, which for each $i\in\{1,\ldots,N\}$, $h\in \Gamma_0$ obeys \[ \Delta(h\cdot^{op}q_i,c)=c_{i,h}=\Delta(\underline{\Pi}h\cdot^{op}q_i,q_1)\]  So since the $q_i$ span $D_{-k}$ over $\Lambda$ and $\Delta$ is bilinear over $K$ (and since for any $\lambda=\sum \lambda_h h\in\Lambda$ the fact that $c\in C_k$ implies that  $\Delta(h\cdot^{op}q_i,c)\neq 0$ for just finitely many $h$ with $\lambda_h\neq 0$), it follows that \[ \Delta(y,c)=\Delta(\underline{\Pi}y,q_1)\] for all $y\in D_{-k}$.  As explained earlier, since $q_1=p$ appears with nonzero coefficient in $d$, since $\underline{\Pi}(d)=d$, and since $\underline{\Pi}|_{V_{\alpha}}=0$, this implies that $c$ defines a class  $a\in H_k(C_{*}^{(\alpha,\infty)})$ with $\underline{\Delta}_{\alpha}(b,a)\neq 0$, completing the proof in the case that $-\alpha\notin \mathcal{A}^{op}(P_{k-1})$.

We now sketch how one can proceed in the case that $-\alpha\in \mathcal{A}^{op}(P_{k-1})$; full details are left to the reader.  First of all note that the assumption (\ref{assumealpha}) was used only to ensure that we could choose an element $x\in D_{-k+1}$ with $\delta x=d$ and $\ell(x)>-\alpha$; if $-\alpha\in\mathcal{A}^{op}(P_{k-1})$ but such an $x$ can still be chosen then the argument above still goes through to produce the desired result.  Thus the only case to address is that in which the only $x$ with $\delta x=d$ have $\ell^{op}(x)=-\alpha$ (since $d$ represents a nontrivial class in $H_{-k}(C_{*}^{(-\infty,-\alpha)})$ we can't have $\ell^{op}(x)<-\alpha$). Let such an $x$ be given.  

Suppose for the moment that the graded filtered Floer-Novikov complex $\frak{c}$ had the property that we could choose $p_1,\ldots p_M\in gr^{-1}(\{k-1\})$ such that the $\pi(p_i)$ are the distinct elements of $\pi(gr^{-1}(\{k-1\}))$ and such
that, instead of (\ref{aalpha}), we have \[ \mathcal{A}^{op}(p_i)=-\alpha\] for each $i$.  The associated isomorphism $\Phi_{\vec{p}}\co D_{-k+1}\to \Lambda^M$ will then satisfy \begin{equation}\label{exact} \barnu(\Phi_{\vec{p}}(x'))=-\alpha-\ell^{op}(x')\end{equation} for all $x'\in D_{-k+1}$. (In particular, $\barnu(\Phi_{\vec{p}}(x))=0$, so the element $g\in \Gamma_0$ from our earlier argument will be the identity.)  Proceeding as above we get a subset $\{x,x'_2,\ldots,x'_m\}\subset D_{-k+1}$ such that $\{d,\delta x'_1,\ldots,\delta x'_m\}$ is a basis for $Im(\delta)$ and such that $\{\Phi_{\vec{p}}(x),\Phi_{\vec{p}}(x'_2),\ldots,\Phi_{\vec{p}}(x'_m)\}$ is a orthonormal basis for the subspace of $\Lambda^M$ which it spans.  However, due to (\ref{exact}), the relation (\ref{ellep}) improves to \[ \ell^{op}\left(\lambda \cdot^{op} x+\sum_{i=2}^{m}\lambda_i\cdot^{op} x'_i\right)\geq \ell^{op}(\lambda \cdot^{op} x).\]  
From this one sees that the map $\underline{\Pi}$ defined in the same way as earlier again vanishes on $Im(\delta|_{D_{-k+1}^{(-\infty,-\alpha)}})$ while satisfying $\underline{\Pi}(d)=d$.  So by constructing $c\in C_k$ so that $\Delta(y,c)=\Delta(\underline{\Pi}y,p)$ for all $y\in D_{-k}$ we again get an element of $H_k(C_{*}^{(\alpha,\infty)})$ pairing nontrivially with $b$.

This completes the proof under the assumption that $\frak{c}$ satisfies the extra assumption at the start of the previous paragraph.  The observation to make now is that, while our given graded filtered Floer-Novikov complex $\frak{c}$ may not satisfy this condition, it can be embedded in one, say $\frak{c}'$, which does.  In the notation of Section \ref{notation}, we have a direct sum decomposition $\Gamma=\Gamma_0\oplus\langle e\rangle$; the group $\Gamma_0$ can be enlarged to a group $\Gamma'_0\geq \Gamma_0$ with an injective homomorphism $\omega'\co \Gamma'_0\to\mathbb{R}$ such that $\omega'|_{\Gamma_0}=\omega|_{\Gamma_0}$ and $\omega'(\Gamma'_0)$ contains $\mathcal{A}(P)$ (note that $\mathcal{A}(P)$ is a finite union of cosets of $\omega(\Gamma)$, so we just need to add finitely many new generators to $\Gamma_0$ to achieve this).  The ``deck transformation group'' for the new   
graded filtered Floer-Novikov complex $\frak{c}'$ will then be $\Gamma'=\Gamma'_0\oplus\langle e\rangle$; the bundle $P\to S$ (with $P=\coprod_{i=1}^{|S|}\Gamma\cdot p_i$) will be replaced by a principal $\Gamma'$-bundle $P'\to S$ with $P'=\coprod_{i=1}^{|S|}\Gamma'\cdot p_i$.  The other ingredients in the definition of a graded filtered Floer-Novikov complex can be adapted from the data for $\frak{c}$ in a straightforward way that is left to the reader.  This produces a chain complex $(C'_*,\partial')$, and also an opposite complex $(\frak{c}')^{op}$ with chain complex $(D'_*,\delta')$.  Importantly, we have $C_*\subset C'_*$, $D_*\subset D'_*$, and the differentials $\partial'$ and $\delta'$ preserve $C_*$ and $D_*$ respectively, restricting to each as the differentials $\partial$ and $\delta$ of the original complexes $\frak{c}$ and $\frak{c}^{op}$.  Also, we can write $C'_*=C_*\oplus E_*$ where $E_*$ consists of formal linear combinations of elements of $P'\setminus P$, and $\partial'$ maps $E_*$ to itself by virtue of the fact that $\partial'$ commutes with the action of the subgroup $\Gamma\leq \Gamma'$.  Likewise, $D'_*=D_*\oplus F_*$ with $\delta'$ mapping $F_*$ to itself.

Now let $0\neq b\in H_{-k}(D_{*}^{(-\infty,-\alpha)})$ and let $d\in D_{-k}^{(-\infty,-\alpha)}$ represent $b$.  
Because $d$ is homologically nontrivial in $D_{*}^{(-\infty,-\alpha)}$ and because $\delta'$ preserves  $F_*$, $d$ is also homologically nontrivial as an element of $D_{*}^{'(-\infty,-\alpha)}$.  Hence the special case that we've already proven finds $c_0\in C'_{k}$ such that $\partial' c_0=0$ mod $C^{'(-\infty,\alpha]}$ and $\Delta(d,c_0)\neq 0$.  Now we can uniquely write $c_0=c+c'$ where $c\in C_{k}$ and $c'\in E_{k}$; because $\partial'$ preserves  $E_*$ and $\partial'|_{C_*}=\partial$ the fact that $\partial'c_0\in C^{'(-\infty,\alpha]} $ implies that $\partial c\in C^{(-\infty,\alpha]}$.  Moreover we clearly have $\Delta(d,c_0)=\Delta(d,c)$.  Thus $c\in C_{k}$ projects to a cycle in $C_{k}^{(\alpha,\infty)}$ pairing nontrivially with our representative $d$ of $b$.  The homology class in $H_k(C_{*}^{(\alpha,\infty)})$ of the projection of $c$ to $C_{k}^{(\alpha,\infty)}$ therefore has nontrivial pairing under $\underline{\Delta}_{\alpha}$ with $b$, as desired.

\begin{remark} The methods above can be modified without much difficulty to show that, similarly, the pairing $H_{-k}(D_{*}^{(-\infty,-\alpha]})\times H_k(C_{*}/C_{*}^{(-\infty,\alpha)})\to K$ induced on homology by $\Delta$ is also nondegenerate; we leave the verification of this to the interested reader.
\end{remark}

\end{document}